\documentclass{amsart}
\usepackage{amssymb, mathrsfs, pb-diagram}
\title{Manifold structures in regular irreducible algebraic monoids}
\author{V. N. Krishnachandran}
\address{Vidya Academy of Science \& Technology, Thrissur - 680 501, Kerala, INDIA}
\email{krishnachandranvn.v.n@vidyaacademy.ac.in}
\newtheorem{lemma}{Lemma}[section]
\newtheorem{thm}[lemma]{Theorem}
\newtheorem{cor}[lemma]{Corollary}
\newtheorem{prop}[lemma]{Proposition}

\theoremstyle{remark}

\newcommand{\fldc}{\ensuremath{\mathbb{C}}}
\newcommand{\fldk}{\ensuremath{\mathbb{K}}}
\newcommand{\rng}[1][e]{\ensuremath{\mathsf{R}(#1)}} 
\newcommand{\nul}[1][e]{\ensuremath{\mathsf{N}(#1)}} 
\newcommand{\gk}[1][k]{\ensuremath{\mathsf{G}_{#1}}} 
\newcommand{\gl}[1][n]{\ensuremath{\mathsf{GL}(#1)}} 
\newcommand{\Gla}{\ensuremath{\mathsf{G}_a^l}}% Grassmann of L-classes in D_a.
\newcommand{\Gra}{\ensuremath{\mathsf{G}_a^r}}% Grassmann of R-classes in
\newcommand{\Gle}{\ensuremath{\mathsf{G}_e^l}}% Grassmann of L-classes in D_a.
\newcommand{\Gre}{\ensuremath{\mathsf{G}_e^r}}

\newcommand{\Rank}{\operatorname{Rank}}
\newcommand{\GL}{\,\mathscr L\,} 
\newcommand{\GR}{\,\mathscr R\,} 
\newcommand{\GD}{\,\mathscr D\,} 
\newcommand{\GJ}{\,\mathscr J\,} 
\newcommand{\GH}{\,\mathscr H\,} 
\newcommand{\Ga}{\Gamma}
\newcommand{\De}{\Delta}
\newcommand{\lro}{\rightarrow}
\newcommand{\Act}[1]{\mathbf{A}_{\mathbf {#1}}}
\newcommand{\Map}[1]{\phi_{#1}}
\newcommand{\Stb}[1]{S_{#1}}
\newcommand{\NAct}[2]{{\Act #1}{}_{,#2}}
\begin{document}
\begin{abstract}
In this paper we study regular irreducible algebraic
 monoids over $\fldc$ equipped with the euclidean
 topology. It is shown that, in such monoids, the Green classes and
 the spaces of idempotents in the Green classes 
 all have natural manifold structures. The
 interactions of these manifold structures and the
 semigroup structures in these monoids have been
 investigated. Relations between these manifolds and
 Grassmann manifolds have been established. A
 generalisation of a result on the dimension of the 
 manifold of rank $k$ idempotents in the semigroup of
 linear endomorphisms over $\fldc$  has been proved. 
\end{abstract}
\maketitle
\section{Introduction}
Let $K$ be an algebraically closed field. 
Let $K[x_1,\ldots,x_m]$ be the algebra of polynomials in
the indeterminates $x_1,\ldots,x_m$ over  $K$.  A subset $X$ of the affine $m$-space $K^m$
is said to be  
algebraic if it is the zero set of a collection of polynomials in
$K[x_1,\ldots,x_m]$ (Definition 1.2 \cite{Renn:05}). The set $X$ is  irreducible 
if it is not a union of two proper algebraic sets. 
A (linear) algebraic semigroup $S = (S, o)$ is an affine variety $S$ along
with an associative product map $o: S \times S \rightarrow S$ which is also a morphism of varieties (Definition 3.1 \cite{Putc:86}).
Identifying the elements of the multiplicative semigroup $M_n(K)$  
of $n\times n$ matrices over $K$ with the elements of $K^{n^2}$, we 
can define algebraic sets in $M_n(K)$. By an algebraic subsemigroup of $M_n(K)$ we mean a subsemigroup of $M_n(K)$ which is also an algebraic subset of the affine $n^2$-space $K^{n^2}$. 
It is known that an algebraic semigroup $S$
 (respectively,  an algebraic monoid $M$)  is isomorphic to an algebraic
subsemigroup (respectively, an algebraic submonoid) of $M_n(K)$ for some $n$ (see Theorem 3.15, Corollary 3.16 \cite{Putc:86}). So, in this paper, by an algebraic semigroup we shall always mean an algebraic subsemigrpup of $M_n(K)$ for some fixed $n$.  Algebraic monoids have been extensively studied by Mohan S. Putcha \cite{Putc:86}, Lex E. Renner \cite{Renn:05}, and others.

A topological semigroup \cite{carr}
is a Hausdorff  space $S$ together with an
associative binary operation  $\centerdot$ such that the product map 
$(x,y) \mapsto x\centerdot y$ is jointly continuous in $x$ and $y$. 
The usual topology on an algebraic set is the Zarisky topology, 
for which the the  closed sets are its algebraic subsets. An algebraic set with the usual topology  is not a Hausdorff space and so an algebraic semigroup cannot be 
a topological semigroup under the usual topology.
However in the special case where $K=\fldc$, the field of complex numbers,  there is another natural topology on algebraic sets, namely, the subspace topology inherited from the euclidean topology on $M_n(\mathbb C)$. 
Since this topology
is Hausdorff,  an algebraic semigroup over $\fldc$ is a  topological semigroup
under this topology (the joint continuity of the 
product map being obvious). 

It will be assumed throughout the rest of this paper that all algebraic semigroups are over the field $\fldc$. It will further be assumed  that all topological terms 
refer to the euclidean  topology. 
In this paper we study the topology of a regular irreducible 
algebraic monoid $M$. The special case where $M=M_n(\mathbb C)$ has been 
studied in \cite{kris:00}. We show here 
that many of the properties of $M_n(\mathbb C)$ 
can be generalised to $M$.

\subsection*{Notations}
For notations and terminology relating to semigroups we have followed \cite{Clif+:61} 
and for those connected with topology of manifolds we
have followed \cite{Dieu:72}. In particular the set of idempotents in a subset $T$ of some semigroup will be denoted by $E(T)$. 
The letter $M$ will always denote a regular
irreducible algebraic submonoid of $M_n(\fldc)$ and $G$ 
the group of units in
$M$. We denote $M_n(\fldc)$ by $M_n$. 
Green's relations in $M_n$ will be denoted by $\GL$, etc.
and the corresponding relations in $M$  by $\GL^M$, etc. The
Green classes in $M_n$ are denoted by $L_a$, etc. and the
corresponding classes in $M$ by $L_a^M$, etc. 
However if $a\in M_n$ and $\Rank(a)=k$ we sometimes denote $D_a$ by $D_k$.

In this paper we  have considered several different (left) actions of groups on sets.
These actions are denoted by $\Act{j}$,  where $j=1,2,\ldots$, the groups and the sets on which the groups act varying with $j$. Let $\Act{j} : X \times Y \rightarrow Y$ denote the action of a group $X$ on a set $Y$. The orbit of $y\in Y$ under the action  $\Act{j}$ is $\Act{j}(X,y)$.
The stabiliser of $y\in Y$ under $\Act{j}$ will be
denoted by $\Stb{j,y}$.
For a fixed element $y\in Y$, the natural map $x\mapsto \Act{j}(x,y)$
from $X$ into $Y$ will be denoted by $\NAct{j}{y}$. 
The canonical map
$x\Stb{j,y}\mapsto \Act{j}(x,y)$ from $X/\Stb{j,y}$ to $\Act{j}(X,y)$ will be denoted by 
$\Map{j,y}$. 

\section{Preliminaries}

We have the following elementary result describing the Euclidean topology of
algebraic sets.

\begin{prop}\label{Prop.vii.1.1}
Let $X\subseteq \fldc^n$ be an algebraic set. Then $X$ is a Hausdorff, locally
compact, $\sigma$-compact space.
\end{prop}
\begin{proof}
Since $\fldc^n$ is Hausdorff, $X$ is
also Hausdorff. Since $\fldc^n$ is locally compact and  since every
algebraic set in $\fldc^n$ is  closed,   $X$ must be  locally
compact.

For  any positive integer $r$  define
$V_r=\{x\in \fldc^n:\|x\|<r\}$ and $U_r=X\cap V_r$. 
Since $\fldc^n=\cup_{r=1}^\infty V_r$, we have $X=\cup_{r=1}^\infty U_r$. We
also have
$
\overline{U_r}=\overline{X\cap V_r}=X\cap \overline{V_r}.
$
Now $\overline{V_r}$ is compact.
Since $X$ is closed in $\fldc^n$, $\overline{U_r}$ 
is also compact in $\fldc^n$. Thus each $U_r$ is
relatively compact set in $X$. Obviously we also have $\overline{U_r}\subseteq
U_{r+1}$. It follows (Theorem XI.7.2 \cite{Dugu:66}) that $X$ is
$\sigma$-compact. 
\end{proof}

Since every algebraic semigroup is an algebraic set, we have the
following corollary to Proposition \ref{Prop.vii.1.1}.

\begin{cor}\label{Cor.vii.1.1}
Every algebraic semigroup is a Hausdorff, locally compact, $\sigma$-compact
space. 
\end{cor} 

A subgroup $A\subseteq \gl$
is said to be algebraic  if $A$ is an algebraic
subset   of \gl.
We have the following  result (Theorem 2.1.2
\cite{Vara:74}) on algebraic groups. 
\begin{thm}\label{Th.vii.1.1}
Let $A$ be an algebraic group in \gl. Then $A$ is a closed analytic subgroup of
\gl. 
\end{thm} 

In general, $G$ is not an algebraic set in $\gl$. However we have
the following result  (see also  Corollary 3.26 \cite{Putc:86}).

\begin{lemma}\label{Lemma.vii.2.1}\ 
\begin{enumerate}
\item
$G=M\cap \gl$.
\item
$G$ is an analytic subgroup of \gl.
\end{enumerate}
\end{lemma}
\begin{proof}
Clearly we have  $G\subseteq \gl$ and so $G\subseteq M\cap \gl$. Let $u\in
M\cap \gl$. Then $\det(u)\ne 0$. Hence (Remark 3.23 
\cite{Putc:86}) we must have $u\,
\GH^M \, 1$. But $G=H_1^M$ and therefore $u\in G$. This proves the first
half of the lemma.

$G$ is obviously a subgroup of \gl. Hence,  to prove that $G$ is an
analytic subgroup of \gl, we need only show that $G$ is a submanifold of
\gl. To prove this,  we consider the set
\begin{equation}\label{Eq.vii.2.1}
G'=\left\{\left[\begin{smallmatrix} \alpha & 0 \\ 0 &
a\end{smallmatrix}\right]: \alpha\in \fldk,a\in M, \alpha\det(a)=1\right\}.
\end{equation}
Since $M$ is an algebraic set, $G'$ is an algebraic set in \gl[n+1].

Now, for any $a\in G$, we write
$
a'= \left[\begin{smallmatrix} \frac {1}{\det(a)} & 0 \\ 0 &
a\end{smallmatrix}\right].
$
Then 
$
G'=\{ a': a\in G\}.
$
For $a,b\in G$, we have $(ab)'=a'b'\in G'$ 
and so  $G'$ is a subgroup of \gl[n+1]. Thus $G'$ is an algebraic
subgroup of \gl[n+1].  Hence, by Theorem \ref{Th.vii.1.1}, $G'$ is a
closed analytic subgroup of \gl[n+1].

The action of $G'$ on the manifold \gl\ defined by
$
(a',u)\mapsto a'\centerdot u=au.
$
is clearly an analytic left action. Hence the map
$
\varsigma\,:\,G' \rightarrow G\, ,\quad a'\mapsto a'\centerdot 1 =a
$
is a subimmersion. Obviously, the map $\varsigma$ is injective and 
therefore ((16.8.8 (iv))  \cite{Dieu:72}) it is an immersion.
Since the map $a\mapsto \det(a)$ is continuous and open
$\varsigma$ is a homeomorphism. Since
it is also an immersion,   ((16.8.4) \cite{Dieu:72})
$G$ must be a submanifold of \gl.
\end{proof}

The following theorems (see \S 16.10 \cite{Dieu:72})
are repeatedly used in the sequel.
\begin{thm}\label{Thm1}
Let an analytic group $X$ act analytically on a manifold $Y$. 
Then the stabiliser 
$S_y$ of $y\in Y$ is an analytic subgroup of $X$.
\end{thm}
\begin{thm}\label{Thm2}
Let $X$ be an analytic group acting analytically on a 
manifold $Y$ under a map $(x,y)\mapsto x\centerdot y$. If $y\in Y$ is such that the orbit $X\centerdot y$ 
is locally closed in $Y$ then $X\centerdot y$ is a submanifold of $Y$ 
and the map $\phi_x:X/S_y \rightarrow X\centerdot y$ defined by 
$xS_y\mapsto x\centerdot y$ is an isomorphism of manifolds.
\end{thm}
%%
%% -------------------------- Section 3 ---------------------------------------
%%
\section{Topology of  Green Classes}

We have the following characterizations of the Green
classes in $M$. These follow from Proposition 6.1 \cite{Putc:86},
Theorem 1.4 \cite{Putc:86} and Proposition II.4.5 \cite{Howi:76}.

\begin{lemma}\label{Lemma.vii.3.1}
Let $a\in M$. Then we have:
\begin{enumerate}
\item
$L^M_a=Ga=L_a\cap M$.
\item
$R^M_a=aG=R_a\cap M$.
\item
$D^M_a=GaG=J^M_a$.
\end{enumerate}
\end{lemma}

The Green classes in $M_n$ are submanifolds of $M_n$. We now prove that the
result is true for $M$ also.

\begin{prop}\label{Prop.vii.3.1}
Let $a\in M$. Then $L_a^M$ and $R_a^M$ are submanifolds of $M_n$.
\end{prop}
\begin{proof}
The map $\Act1' :\gl\times M_n \lro M_n$ defined by 
$(v,x) \mapsto  vx$
is an analytic left action of $\gl$ on $M_n$.
Since $G$ is a submanifold of $\gl$  (see  Lemma \ref{Lemma.vii.2.1}),
$G\times M_n$ is a submanifold of
$\gl\times M_n$ and so the map
$$
\Act1 : G\times M_n \lro M_n, 
\quad 
(v,x) \mapsto  vx,
$$ 
which is the restriction of $\Act1'$ to $G$, is also analytic
((16.8.3.4) \cite{Dieu:72}).

Now let $a\in M$. Lemma \ref{Lemma.vii.3.1} implies that $\Act1(G,a)$, the orbit of
$a$ under $\Act1$, is $L_a^M=L_a\cap M$.
It is known that $L_a$ is a submanifold of $M_n$ (Proposition 2.2 \cite{kris:00}).
Therefore, $L_a$ is a locally compact subspace of $M_n$. Since
$M_n$ is itself locally compact, we can find an open set $U$ and
a closed set $F$ in $M_n$ such that $L_a=U\cap F$. Therefore,
$L_a^M=U\cap(F\cap M)$. Since $M$ is an algebraic set in $M_n$,
it is a closed set in $M_n$. Thus, $L_a^M$ is the intersection of
an open set and a  closed set in $M_n$. It follows that $L_a^M$
is locally closed in $M_n$ ((12.2.3) \cite{Dieu:70}).
Now,  it follows that (see \S 16.10 \cite{Dieu:72}) that $L_a^M$ is a submanifold
of $M_n$.

To prove the result concerning $R^M_a$ we consider the map 
$$
\Act2 : G\times M_n \lro M_n, 
\quad  
(w,x) \mapsto  xw^{-1}.
$$
\end{proof}
Since $L_a$ and $L_a^M$ are submanifolds of $M_n$, and since
$L_a^M\subseteq L_a$ (and similarly for $R_a^M$) 
we have ((16.8.6.1) \cite{Dieu:72}):
\begin{cor}
$L_a^M$ is a submanifold of $L_a$ and $R_a^M$ is a submanifold of $R_a$.
\end{cor}

Since $L_a^M$ and $R_a^M$ are locally closed subspaces of $M_n$ we also have (see Theorem \ref{Thm2}):
\begin{cor}\label{Cor.vii.1}
For $i=1,2$ we have:
\begin{enumerate}
\item
$\Stb{i,a}$ is an analytic subgroup of $G$.
\item
$\Map{i,a}$ is an isomorphism of manifolds.
\item
$\NAct{i}{a}$  is a continuous open surjection.
\end{enumerate}
\end{cor}

We next consider the manifold structure of $\GD$-classes of  $M$. 
\begin{lemma}\label{Lemma.vii.3.3}
Let $a\in M$. Then
$D^M_a$ is a locally closed  subspace of $M_n$.
\end{lemma}
\begin{proof}
By Lemma \ref{Lemma.vii.3.1}, we have $D^M_a=GaG$.
Let $F=\overline{MaM}$ be the Zarisky-closure of $MaM$ in $M_n$.
For any $u\in G$, the map $y\mapsto uy$ is a Zarisky-homeomorphism 
 of $M_n$ onto itself. Therefore,
$
uF=u\overline{MaM}=\overline{uMaM}.
$
But, clearly, we have $uM=M$ so that $uMaM=MaM$. Hence,
$uF=\overline{MaM}=F$. Similarly, we also have $Fu=F$.

Now, there exists a Zarisky-open subset $U_0$  of $F$ such that
$U_0\subseteq GaG$
(see proof of Proposition 6.1 in \cite{Putc:86}).
Since $U_0$ is a Zarisky open subset of $F$, we can find a Zarisky-open
set $X_0$ in $M_n(\fldk)$ such that
$
U_0=F\cap X_0.
$

Choose some fixed element $u_0av_0\in U_0$. Consider any element  $b=uav$ in
$GaG$. We write
$ U_b=u'U_0v'$ where $ u'=uu_0^{-1}$ , $ v'=v_0^{-1}v$.
We show that $U_b= F\cap (u'X_0v')$.

Let $x\in U_0$ then $x\in F$ and $x\in X_0$.
Since $u'F=F$, we have $u'x\in F$.  Now $u'x\in u'X_0$. Therefore
$u'U_0\subseteq F \cap (u'X_0)$.  To prove the inclusion in the
reverse direction, let $x\in F\cap (u'X_0)$. Let $x=u'x_0$ for some
$x_0\in X_0$. Then
$
x_0=(u')^{-1}x \in (u')^{-1} F =  F.
$
This shows that $x=u'x_0\in u'(F\cap X_0)$ implying that
$F\cap (u'X_0)\subseteq u'U_0$. Hence $F\cap
(u'X_0)=u'U_0$. By a similar argument we can now show that $F\cap
(X_0v')=U_0v'$. 
Now   we  have
\begin{align*}
U_b     & = u'U_0v' \\
        & = (u'U_0)\cap (U_0v')\\
        & = (F\cap u'X_0)\cap(F\cap X_0v')\\
	& = F\cap (u'X_0\cap X_0v')\\
	& = F\cap (u'X_0v'). 
\end{align*}

Since $b\in U_b=F\cap u'X_0v' \subseteq GaG $  we have
$$
GaG=F\cap \left[ \cup_{b\in GaG} (u'X_0v')\right].
$$
Since $F$ is Zarisky-closed, $F$ is a closed set in $M_n$.
Since $X_0$ is Zarisky-open it is  open in $M_n$. Hence $u'X_0v'$ is
also open. Therefore 
$\cup_{b\in GaG} (u'X_0v')$ is open in $M_n$. This shows that
$GaG$ is the intersection of a closed and an open set in $M_n$.
Therefore $GaG$ is a locally closed set in $M_n$
((12.2.3) \cite{Dieu:70}).
\end{proof}
\begin{prop}\label{Prop.vii.3.2}
Let $a\in M$. Then $D^M_a$ is a  submanifold of $M_n$.
\end{prop}
\begin{proof}
Consider the map
$$
\Act3 :(G\times G)\times M_n \lro M_n,
\quad
((v,w),x)\mapsto vxw^{-1}.
$$
This defines an analytic left action of $G\times G$ on $M_n$.
By Lemma \ref{Lemma.vii.3.1}, $\Act3((G\times G), a)=GaG=D^M_a$.
Since  $GaG$ is a locally closed subspace of
$M_n$ (Lemma \ref{Lemma.vii.3.3}),
$D^M_a$ is a submanifold of $M_n$ (Theorem \ref{Thm2}).
\end{proof} 

Since $D_a$ is a submanifold of $M_n$, and since $D_a^M\subseteq
D_a$ we have:

\begin{cor}
$D_a^M$ is a submanifold of $D_a$.
\end{cor}

Theorem \ref{Thm2}  implies the following:
\begin{cor}\label{Cor.vii.Map3}         
$\Stb{3,a}$ is an analytic subgroup of
$G\times G$ and $\Map{3,a}$
is an isomorphism of manifolds.
\end{cor}
%
%  ------------------------------ Section 4 ----------------------------------
%
\section{Green Classes  and Grassmann Manifolds}
Let $V$ be a fixed $n$-dimensional vector space over $\fldc$ and let 
$\gk$ be the set of all $k$-dimensional subspaces of $V$. Chosing some
fixed ordered basis $B$ we can represent elements of $V$ by 
coordinate vectors relative to $B$. Also elements of $M_n$
can be looked upon as matrix representations of liear endomorphisms
of $V$ relative to $B$. If $M(k,n;k)$ denotes the space of 
$k \times n$ matrices of rank $k$, then the topology of
$\gk$ is the quotient topology induced by the map
$q:M(k,n;k)\rightarrow \gk$ which maps $P\in M(k,n;k)$
to the subpace of $V$ generted by the rows of $P$. This topology makes
$\gk$ a manifold called the Grassmann Manifold.
If $x\in M\subseteq M_n$ and $\Rank(x)=k$ then the range $\rng[x]$ is in $\gk$
and the null space $\nul[x]$ is in $\gk[n-k]$. 

\begin{lemma}\label{Lemma.vii.4.1}
Let $a,b\in M$. Then
\begin{enumerate}
\item
$a\,\GL^M\, b\,\Leftrightarrow \, \rng[a]=\rng[b].$
\item
$a\,\GR^M\, b\,\Leftrightarrow \, \nul[a]=\nul[b].$
\item
$\GD^M=\GJ^M$.
\end{enumerate}
\end{lemma}
\begin{proof}
Since $M$ is a regular subsemigroup of $M_n$, for any $a,b\in M$,
$a \,\GL^M\, b$ if and only if $a\,\GL\, b$
(Proposition II.4.6 \cite{Howi:76}). But $a\,\GL\, b$
if and only if $\rng[a]=\rng[b]$ (see p.57 in \cite{Clif+:61}). This proves the claim regarding $\GL^M$.
The proof of the claim regarding $\GR^M$   is similar.

The result $\GD^M=\GJ^M$ follows from Theorem 1.4 \cite{Putc:86}.  
\end{proof}

For $a\in M$, we  write 
\begin{equation}\label{Eq.vii.Gla}
\Gla=\{\rng[x]: x\in D^M_a\},\quad
\Gra=\{\nul[x]: x\in D^M_a\},
\end{equation}
Note that if
$\Rank(a)=k$ then $\Gla\subseteq \gk$ and $\Gra \subseteq \gk[n-k]$.
We begin with the following lemma which shows that, in the definition
of $\Gla$ [$\Gra$], we need consider only just one $\GR^M$- [$\GL^M$-] 
class contained in $D_a^M$.
\begin{lemma}\label{Lemma.vii.4.2}
If $b\in D^M_a$ then
\begin{enumerate}
\item
$\Gla\,=\,\{\rng[x]\,:\, x\in R^M_b\}$.
\item
$\Gra\,=\,\{\nul[x]\,:\, x\in L^M_b\}$.
\end{enumerate}
\end{lemma}
\begin{proof}
Let $W\in \Gla$. Then $W=\rng[y]$ for some $y\in D^M_a$. Choose $x\in R^M_b\cap
L^M_y$. Then $\rng[x]=\rng[y]=W$ and so $W\in \{\rng[x]\,:\, x\in R^M_b\}$.
Since the reverse inclusion is obvious we have the result concerning $\Gla$.
The proof of the result concerning $\Gra$ is similar.
\end{proof}

The main result of this  section is that if $\Rank(a)=k$ then $\Gla$ [$\Gra$] is a submanifold of
$\gk$ [$\gk[n-k]$]. We begin with the following lemma.
\begin{lemma}\label{Lemma.4.3}
Let The map $\Gamma_k:D_k\rightarrow \gk$ {\em [}$\Delta_k:D_k\rightarrow \gk[n-k]${\em ]} defined by $x\mapsto \rng[x]$ {\em [}$x \mapsto \nul[x]${\em ]}
is a continuous open surjection.
\end{lemma}
\begin{proof}
Let $U$ be an open set in \gk.
Let $x_0\in \Ga_k^{-1}(U)$,  $\Ga_k(x_0)=W\in U$ and $\nul[a]=N$.
Choose an ordered  basis $B$ for $V$ such that the last 
$(n-k)$-vectors
in $B$ form an ordered basis for \nul[a]. 
Relative to $B$,
$x_0=\left[\begin{smallmatrix}X_0\\O\end{smallmatrix}\right]$ 
where $X_0\in M(k,n;k)$. 
Since $q^{-1}(U)$ is open in
$M(k,n;k)$ and $X_0\in q^{-1}(U)$
we can find $\epsilon>0$ such that 
$$
B_M(X_0,\epsilon)=\{X\in M(k,n;k)\, :\, \|X-X_0\| < \epsilon\}
\subseteq q^{-1}(U).
$$
Let $ B(x_0,\epsilon)=\{x\in D_k\,:\, \|x-x_0\| < \epsilon\} $.
If $x\in B(x_0,\epsilon)$ and $X$ is the matrix formed 
by the first $k$
rows of $x$  then  $\|X_0-X\| <\epsilon$ so that 
$X\in B_M(X_0,\epsilon)$.
Since $X\in M(k,n;k)$, the rank of $X$ is $k$. 
Since $\Rank(x)$ is also
$k$, the range $\rng[x]$ of $x$ is the subspace spanned 
by the rows of
$X$. Thus we have $\rng[x]=q(X)$. Therefore,
$
\Gamma_k(x)=\rng[x]=q(X)\in U.
$
Hence  we have
$
B(x_0,\epsilon)\subseteq \Ga_k^{-1}(U).
$
This shows that $\Ga_k^{-1}(U)$ is open.
It follows that $\Ga_k$ is continuous.

To show that $\Ga_k$ is open, let $U'$ be an open set in $D_k$. Now
$
\Ga_k^{-1}(\Ga(U'))=\cup_{x\in U'} L_x.
$
Also, $y\in \cup_{x\in U'} L_x$ if and only if
$y\GL x$ for some $x\in U'$. Also, $y\GL x$ if and only if $y=ux$
for some $u\in \gl$.
Thus, $y\in \cup_{x\in U'} L_x$ if and only if $y=ux$
for some $x\in U'$ and for some $u\in\gl$.
Therefore we have
$ \Ga_k^{-1}(\Ga_k(U'))=\cup_{u\in \gl} uU'$.
The map $x\mapsto ux$ is a homeomorphism of $M_n$ onto itself.
Since $U'$ is open in $D_k$, $uU'$ is open in $D_k$ 
for every $u\in \gl$.
Therefore $\Ga_k^{-1}(\Ga_k(U'))$ is a union of open sets 
in $D_k$ and
therefore is itself open. This shows that $\Ga_k$ is an open map.

In a similar way one can show that the map $\Delta_k:D_k\rightarrow \gk[n-k]$
defined by $x \mapsto \nul[x]$ is also a continuous open surjecion.
\end{proof}
\begin{prop}\label{Th.iii.4.3}
The restrictions $\Ga_k|R_a$ and 
$\De_k|L_a$ are continous open surjections. 
\end{prop}
\begin{proof}
For any open set $U\subseteq R_a$, we have
$ \Ga_k^{-1}(\Ga_k(U))=\cup_{x\in U} H_x$.
Now let $e\in E(R_a)$. For all $c\in H_e$, the map $\lambda_c:x\mapsto cx$ is a bijection of $R_a$
onto itself with inverse $\lambda_{c'}$ where $c'$ is the unique inverse of $c$
in $H_e$. Since these translations are continuous  the
map $\lambda_c$ is a homeomorphism  and so $\lambda_c(U)=cU$ is an open set in $R_a$.
Further, since $H_b=H_eb$ for all $b\in R_a$, we have
$ \Ga_k^{-1}(\Ga_k(U))= \cup_{c\in H_e} cU $
which is open in $R_a$. Since $\Ga_k$ is a  quotient map, $\Ga_k(U)$ is open in 
\gk. This proves that $\Ga_k$ is an open map.  Since $\Ga_k$ is a quotient
map, clearly it is continuous. The proof is similar for $\De|L_a$.
\end{proof}

We now define the following maps:
\begin{alignat}{2}
\Ga_a^M\,& :\,D_a^M\rightarrow \Gla\, ,\quad     && x\mapsto
\rng[x],\label{Eq.vii.4.1'}\\
\De_a^M\,& :\,D_a^M\rightarrow \Gra\, ,\quad     && x\mapsto
\nul[x].\label{Eq.vii.4.2'}
\end{alignat}
Note that if $a\in D_k$ then $\Ga_a^M=\Ga_k|D_a^M$ and $\De_a^M=\De_k
|D_a^M$.
We now have the following result.
\begin{lemma}\label{Th.vii.4.1}
The map  $\Ga_a^M$  {\em [}$\De_a^M${\em ]} is a continuous  open surjection. Hence
the topology of $\Gla$ {\em [}$\Gra${\em ]} is the quotient topology induced by $\Gamma_a^M$ {\em [}$\Delta_a^M${\em ]}.
\end{lemma}
\begin{proof}
By Lemma \ref{Lemma.vii.4.1}, for any $a,b\in M$ we have $a\GL^M b$ if
and only if $\rng[a]=\rng[b]$.  Hence the partition of $D_a^M$ induced
by $\Ga_a^M$ is the partition of $D_a^M$ into $\GL$-classes.

By Lemma  \ref{Lemma.4.3}, the map $\Ga_k$ is continuous. Hence its
restriction to $D_a^M$, which is the map $\Ga_a^M$, is also continuous.
It is easy to see that if, in the proof of Lemma  \ref{Lemma.4.3},
in the proof of the fact that $\Ga_k$ is open
 we replace $M_n(\fldk)$ by $M$ and \gl\ by
$G$, then the resulting argument is still valid. It follows that the map
$\Ga_a^M$ is open also.
\end{proof}

The analog of Proposition \ref{Th.iii.4.3} is also valid. We state the
result without proof.
\begin{lemma}\label{Th.vii.4.1'}
The restriction 
$
\Ga_a^M|R_a^M:R_a^M\lro \Gla
$
 {\em [}$\De_a^M|L_a^M:L_a^M\lro \Gra ${\em ]}
is a continuous open surjection.
Therefore,  the topology of $\Gla$ {\em [}$\Gra${\em ]} is the
quotient topology induced by this map.
\end{lemma}

We now consider the map
\begin{equation*}\label{Eq.vii.Act4}
\Act4: G\times \gk \lro \gk\, ,\quad(u,W)\mapsto Wu^{-1}
\end{equation*}
which defines a left action of $G$ on \gk.
By introducing local coordinates in \gk\  we can see that
this action  is analytic. Theorem \ref{Thm1} implies that, for $W\in \gk$,
$\Stb{4,W}$  is an analytic subgroup of $G$.
\begin{prop}\label{Prop.4.7}
For $e\in E(M)$, we have $\Stb{4,\rng[e]} = \{u\in G : eue=eu\}$.
\end{prop}
\begin{proof}
Let $u$ be an arbitrary element in $G$.
\begin{alignat*}{2}
u\in \Stb{4,\rng[e]}\,& \Leftrightarrow\, \rng u^{-1}=\rng
                        && \quad\text{by definition of $S_W$}\\
                  & \Leftrightarrow\, \rng=\rng u
                        &&\\
                  & \Leftrightarrow\, \rng=\rng[eu]
                        &&\\
                  & \Leftrightarrow\, e\, \GL\, eu
                        && \\
                  & \Leftrightarrow\, ve=eu\quad\text{for some}\quad v\in G
                        &&\quad\text{(Lemma  \ref{Lemma.vii.3.1})}
\end{alignat*}
Now if $eu=ve$ then we have $eue=eu$. Conversely, if $eue=eu$ then
$e=e(ueu^{-1})$ which implies that $e\,\GL^M\, ueu^{-1}$ and
so, by Lemma
\ref{Lemma.vii.3.1}, we can find $w\in G$ such that $e=w(ueu^{-1})$. Taking
$v=wu$ we get $eu=ve$.
Therefore
$
\Stb{4,\rng[e]}=\{u\in G\,:\, eue=eu\}.
$
\end{proof}

In the terminology and notations of \cite{Putc:86} (see p.48 \cite{Putc:86}),  
$\Stb{4,\rng}$
is the {\em left centralizer} $C_G^l(e)$   of $e$ in $G$. 
\begin{thm}\label{Th.vii.4.1a}
Let $a\in M$, $\Rank(a)=k$. Then $\Gla$ is a submanifold of  $\gk$. 
\end{thm}
\begin{proof}
 Let $W_0=\rng[a]$. Then  $\Stb{4,W_0}$
is an analytic subgroup of $G$. Hence the orbit manifold $G/\Stb{4,W_0}$
exists (see (16.10.6) in \cite{Dieu:72}).

\medskip
\[
\begin{diagram}
\node{G}\arrow{e,t}{\NAct{2}{a}} \arrow{s,l}{\pi} \arrow{se,t}{\NAct{4}{W_0}} \node{R_a^M} 
\arrow{s,r}{\Gamma_a^M} \node{} \\
\node{G/\Stb{4,W_0}}\arrow{e,b}{\Map{4,W_0}}\node{\Gla} \arrow{e,b}{\subseteq}\node{\mathsf{G}_{k}}
\end{diagram}
\]
\medskip

Let $\pi  : G \rightarrow  G/\Stb{4,W_0}$ be defined by 
$ u\mapsto u\Stb{4,W_0}$. Then we have 
$\NAct{4}{W_0}=\pi\circ \Map{4,W_0}$.

Since $\Act4$ is analytic action of the analytic group $G$ on the
analytic manifold \gk, the map $\NAct{4}{W_0}$ is a subimmersion
(see (16.10.2) in \cite{Dieu:72}). Therefore, the map $\Map{4,W_0}$ is
also a subimmersion (see (16.10.4) in \cite{Dieu:72}). Since
$\Map{4,W_0}$ is injective also, it is an immersion (see (16.8.8(iv)) in
\cite{Dieu:72}).

Now $\Map{4,W_0}(G/\Stb{4,W_0})=\Gla$.
Hence, to show that $\Gla$ is a submanifold of \gk\ it is enough to show
that (see (16.8.4) in \cite{Dieu:72}) the map $\Map{4,W_0}$ is a
homeomorphism from $G/\Stb{4,W_0}$ onto the subspace $\Gla$ of \gk.

To prove that $\Map{4,W_0}$ is a homeomorphism, 
we note that, we have
$\NAct{4}{W_0}=\NAct{2}{a}\circ\Ga_a^M$.

By Lemma \ref{Lemma.vii.4.1}, $\Ga_a^M$ is a continuous open surjection
from $G$ onto $R_a^M$.
By Corollary \ref{Cor.vii.1} the map  $\NAct{2}{a}$
is  a continuous  open surjection from $R_a^M$ onto $\Gla$.
Therefore the map $\NAct{4}{W_0}$ is also a continuous open surjection
from $G$ onto $\Gla$.
 Hence the topology on $\Gla$ is the quotient topology
induced by the map $\NAct{4}{W_0}$. The quotient space induced by this map is
$G/\Stb{4,W_0}$.
Hence $\Map{4,W_0}:G/\Stb{4,W_0}\lro \Gla$, which is the quotient
map induced by the map $\NAct{4}{W_0}$, must be a homeomorphism.

The proof of the theorem is now complete.
\end{proof}

From Proposition \ref{Prop.4.7} we have the following corollary (see (16.8.4) in
\cite{Dieu:72}).
\begin{cor}\label{Cor.vii.4.1a}
Let $e\in E(M)$. Then  $C^l_G(e)$  is an
analytic subgroup of $G$ and $\Map{4,\rng}$ 
is an isomorphism of manifolds.
\end{cor}

We have corresponding results involving $\GL$-classes. In this case we
consider the  action 
$$
\Act5:G\times \gk[n-k]\rightarrow \gk[n-k],
\quad 
(u,N)\mapsto Nu^{-1}.
$$
% -------------------------------- Section 5 ----------------------------------
%
\section{The Space of Idempotents}

We begin with the following lemma which characterizes the restrictions
of Green's relations to
the set of idempotents in $M$
(see Corollary 6.8 in \cite{Putc:86}). Note that the restrictions of $\GL$ and
$\GR$ to $E(M)$ are the biorder relations $L$ and $R$ in $E(M)$.
\begin{lemma}\label{Lemma.vii.5.1}
Let $e,f\in E(M)$. Then we have:
\begin{enumerate}
\item
$e\,\GD^M\, f\,\Leftrightarrow\, f=ueu^{-1}\quad\text{for some}\quad u\in
G$. 
\item
$e\,\GL^M\, f\,\Leftrightarrow\, f=ueu^{-1}=ue\quad\text{for some}\quad
u\in G$.
\item
$e\,\GR^M\, f\,\Leftrightarrow\, f=ueu^{-1}=eu^{-1} \quad\text{for
some}\quad u\in G$.
\end{enumerate}
\end{lemma}

Lemma \ref{Lemma.vii.5.1} has the following consequence. 
\begin{prop}
Any two $\GL^M$- [$\GR^M$-] classes contained
in the same $\GD^M$-class of $M$ are
homeomorphic under a  conjugation.
Moreover, this homeomorphism  preserves the idempotents. 
\end{prop}
\begin{proof}
Let $b,c\in D_a^M$ and let  $f\in E(L_b^M)$ and $g \in E(L_c^M)$. 
Since $f\GD^M g$, by Lemma \ref{Lemma.vii.5.1},  we can find $u\in G$ such that $g=ufu^{-1}$.

Now, let $x \in L_b^M$ and $y= uxu^{-1}$. Since 
$x\GL^M f$,  by Lemma \ref{Lemma.vii.3.1},  we can find $v\in
G $  such that $x=vf$. Now we have
$$
y=uxu^{-1}=uvfu^{-1}=uvu^{-1}(ufu^{-1})=(uvu^{-1})g.
$$
This  shows that $y\GL^M g$ and so $y\GL^M c$. Conversely, if $y\in L_c^M$ then it
can be shown that $u^{-1}yu\in L_b^M$. Therefore the map $x\mapsto uxu^{-1}$ is a
bijection from $L_b^M$ onto $L_c^M$. This map is clearly a homeomorphism.
It is also
a conjugation. Thus $L_b^M$ and $L_c^M$ are homeomorphic under a conjugation. This homeomorphism obviously preserves idempotents. 

The proof of the result regarding $\GR^M$-classes is similar.
\end{proof}
Let $e\in E(M)$. To discuss the topology of $E(D_e^M)$, we consider the action $\Act6$ of $G$ on $E(D_e^M)$ defined by 
\begin{equation*}\label{Eq.vii.Act6}
\Act6: G\times E(D_e)\lro E(D_e),\quad
(u,f)\mapsto ufu^{-1}.
\end{equation*}
As in the proof of Proposition \ref{Prop.vii.3.1}, we can see that
this action is  analytic and, by Lemma \ref{Lemma.vii.5.1}, we have
$E(D^M_e)= \Act6(G,e).$ 
We also have
$$
\Stb{6,e} = \{u\in G: ueu^{-1}=e\}  = \{u\in G :ue=eu\}.
$$
In the notations and
terminology of \cite{Putc:86} (see p.48 \cite{Putc:86}), 
$\Stb{6,e}$ is the {\em centraliser} $C_G(e)$ of $e$ in $G$.
\begin{prop}
Let $e\in E(M)$. Then 
$E(D^M_e)$ is a  submanifold of $E(D_e)$. Moreover, 
the map $\Map{6,e}$
is an isomorphism of manifolds.
\end{prop}
\begin{proof}
By Lemma \ref{Lemma.vii.4.1} we have  $\GD^M=\GJ^M$. Hence 
Proposition 5.8 in \cite{Putc:86} implies that
$E(D^M_e)$ is an irreducible algebraic subset of $M$
and hence it is itself an algebraic
set. Therefore $E(D^M_e)$ is a closed set in $M_n(\fldk)$.
Since $E(D_e)$ is a closed set in $M_n(\fldk)$ (see \cite{kris:00})  and since $E(D_e^M)$ is
also a closed set in $M_n(\fldk)$, $E(D_e^M)$ is a closed set in
$E(D_e)$. Hence $E(D_e^M)$ is a locally closed set in $E(D_e)$.

Since $E(D^M_e)=G\centerdot e$ under the action 
$\Act6$ of $G$ on $E(D_e)$, 
Theorem \ref{Thm2} implies that $E(D_e^M)$ is a submanifold
of $E(D_e)$.

Since $G$ is an analytic group and since  $\Act6$ is an
analytic action on the manifold $E(D_e)$, the stabiliser $C_G(e)$ is
an analytic subgroup of $G$. Theorem \ref{Thm2} also implies that the map $\Map{6,e}$ is an isomorphism of manifolds.
\end{proof}
Since $E(D^M_a)$ is a closed set in $M$, it is a closed
set in $E(M)$ also. The family  of $\GD^M$-classes in $M$ is a finite family (see Theorem 5.10 in \cite{Putc:86}). So the
family $\{E(D^M_e)\,:\, e\in E(M)\}$, the union of whose members is
$E(M)$,
contains only a finite number of mutually disjoint,  closed
subsets of $E(M)$. 

%
%\begin{rem}
Let $e,f\in E(M)$ and $\Rank(e)=\Rank(f)$. Then $E(D_e)=E(D_f)$. But $E(D^M_e)$
and $E(D^M_f)$ need not be homeomorphic. They need not even be manifolds of the
same dimension. For example, consider the monoid
$$
M=\left\{\begin{bmatrix}
            \alpha & 0 \\ 0 & A
         \end{bmatrix}
         :\alpha\in \fldk,A\in M_2(\fldk)\right\}.
$$
The group of units of this monoid is
$$
G=\left\{\begin{bmatrix}
            \alpha&0\\0&A
         \end{bmatrix}
         :\alpha\det(A)\ne 0\right\}.
$$
Let $e=\left[\begin{smallmatrix}
              1 & 0 \\ 0 & O'
             \end{smallmatrix}\right]$ where
$O'\in M_2(\fldk)$ is the zero-matrix and
$f=\left[\begin{smallmatrix}
             0 & 0 \\ 0 & f'
         \end{smallmatrix}\right]$ where
$f'=\left[\begin{smallmatrix}
             1 & 0 \\ 0 & 0
          \end{smallmatrix}\right]\in M_2(\fldk)$.
Then $e,f\in E(M)$ and $\Rank(e)=\Rank(f)=2$. Now we can easily verify that
$$
E(D^M_e)=\{e\}\quad\text{and}\quad
E(D^M_f)=\left\{\begin{bmatrix}
                      0&0\\0&g'
                \end{bmatrix}:g'\in
                       E(D_{f'})\subseteq M_2(\fldk)\right\}.
$$
Obviously $\dim E(D^M_e)=0$. Since $\Rank(f')=1$, by Theorem 5.3 \cite{kris:00}, we have $ \dim E(D_f^M)= E(D_{f'})=2$.
%\end{rem}
%

It is known that if $e\in E(M)$ and $\Rank(e)=k$, then
$E(L_e)$ and $E(R_e)$ are affine spaces of dimension $k(n-k)$ (Proposition 3.1\cite{kris:00}) . We show
here that, if $e\in M$, then $E(L_e^M)$ and $E(R_e^M)$ are submanifolds
of these affine spaces. We begin with the following lemma.
\begin{lemma}\label{LemmaELeM}
If $e,f\in E(M)$ and $e\GL^M f$ then $E(L_e^M)=\{ufu^{-1}:u\in C_G^l(e)\}$.
\end{lemma}
\begin{proof}
Since $e\GL^M f$, by Lemma \ref{Lemma.vii.4.1},
$\rng=\rng[f]$ and so by  Proposition \ref{Prop.4.7} we have $C_G^l(e)=C_G^l(f)$. 

Let $u\in C_G^l(e)$ and $g=ufu^{-1}$. We show that $g\in E(L_e^M)$. $g$ is clearly in $E(M)$. Since $C_G^l(e)=C_G^l(f)$ we have $fuf=fu$. Since $C_G^l(f)$ is a group, $u^{-1}\in C_G^l(f)$ and so $fu^{-1}f=fu^{-1}$. From these we get $f(ufu^{-1}) =f$ and $(ufu^{-1}) f = ufu^{-1}$, that is, $fg=f$ and $gf=g$.  Therefore $g\GL^M f$. 

Let $g$ be an arbitrary element in $E(L_e^M)$. Then $g\GL^M f$, and by Lemma \ref{Lemma.vii.5.1}, we can find $u\in G$ such
that $g=ufu^{-1}=uf$. For this $u$, we have $f(ufu^{-1})=fg = f$ and so
$fuf=fu$ implying that $u\in C_G^l(f)=C_G^l(e)$. Thus $g=ufu^{-1}$ for some $u\in C_G^l(e)$.
\end{proof}
We next consider the following action of $C_G^l(e)$ on $E(L_e)$.
\begin{equation*}\label{Eq.vii.Act7}
\Act7:C_G^l(e)\times E(L_e)\lro E(L_e)\, ,\quad (u,f)\mapsto ufu^{-1}.
\end{equation*}
\begin{lemma}\label{LemmaA7}
Let $e\in E(M)$. 
\begin{enumerate}
\item
The action $\Act7$ is analytic.
\item
$\Stb{7,e}=C_G(e)$.
\item
$\Act7(C_G^l(e),e) = E(L_e^M)$.
\end{enumerate}
\end{lemma}
\begin{proof}
That $\Act7$ is analytic is obvious. Obviously $\Stb{7,e}\subseteq C_G(e)$. If $u\in C_G(e)$ then $ueu^{-1}=e$ and so $eue=eu$, and hence $u\in C_G^l(e)$.  Therefore $u\in \Stb{7,e}$ implying that $C_G(e)\subseteq \Stb{7,e}$. Item (3) follows from Lemma \ref{LemmaELeM}.
\end{proof} 
\begin{thm}\label{Th.vii.5.1a}
Let $e\in E(M)$. 
\begin{enumerate}
\item
$E(L^M_e)$   is a  submanifold  of $E(L_e)$ and also of $E(D^M_e)$.
\item
$C_G(e)$ is an analytic subgroup of $C^l_G(e)$.
\item
$\Map{7,e}$ is an isomorphism of manifolds.
\end{enumerate}
\end{thm}
\begin{proof}
It is known that 
$E(L^M_e)$ is an irreducible algebraic subset
of $M$ (Proposition 5.8 \cite{Putc:86}).  Hence it is itself an algebraic set.
So, by Lemma \ref{Prop.vii.1.1},
it is a locally compact space and hence it is locally closed also. Since $E(L_e^M)$ is the orbit of $e$ under the action $\Act7$, by Theorem \ref{Thm2},
$E(L^M_e)$ is a submanifold of $E(L_e)$. Theorem \ref{Thm2} also implies that  $\Map{7,e}$ is an isomorphism of manifolds.
\end{proof}

We have a corresponding result involving $\GR$-classes.
In this case we
consider the  action 
$$
\Act8:C_G^r(e) \times E(R_e) \rightarrow E(R_e),
\quad 
(u,f) \mapsto ufu^{-1}.
$$
%
%--------------------------------------------------------------------------
%

We have some results involving the dimensions of the various manifolds
considered in this section.
\begin{thm}\label{Th.vii.5.2}
Let $e\in E(M)$. Then
\begin{enumerate}
\item
$ \dim E(L^M_e)=\dim \Gre$.
\item
$\dim E(R^M_e)=\dim \Gle$. 
\end{enumerate}
\end{thm}
\begin{proof}
Let $\rng[e]=W$. Then   $U_W=\{\rng[f]:f\in E(R_e)\}$ is an
open set in \gk\ and the map $\Ga_k^e:E(R_e)\rightarrow U_W$
defined by $f\mapsto \rng[f]$ is a homeomorphism (see Theorem 4.4 \cite{kris:00}).

Next, let $U^M_W=\{\rng[f]:f\in E(R^M_e)\}$. Then the restriction
$\Ga_k^e\big|E(R^M_e)$ is a homeomorphism of $E(R^M_e)$ onto $U^M_W$.

We now show that $U^M_W=U_W\cap \Gle$ so that $U^M_W$ is an open set in $\Gle$.
Let $W'\in U_W\cap \Gle$. Then $W'=\rng[f]=\rng[a]$ for some $f\in E(R_e),a\in
R^M_e$. Since $\rng[f]=\rng[a]$, we have   $f\GL a$ in $M_n(\fldk)$.
Since $a\,\GR\,e$ in $M$ we have $a\,\GR\,e$ in $M_n(\fldk)$
(see Proposition II.4.5 \cite{Howi:76}).
We also have $f\,\GR\,e$ in $M_n(\fldk)$. Thus
$f\,\GH\,a$ in $M_n(\fldk)$. Since $f$ is an idempotent, $H_a$ is a
subgroup of $M_n(\fldk)$ with identity $f$. Therefore $H^M_a$ is a subgroup of
$M$  and $H^M_a\subseteq H_a$. This implies that $f\in H^M_a$
and so $f\,\GR\,a\,\GR\,e$ in $M$. Hence $W'=\rng[f]\in  U^M_W$.
Therefore $U_W\cap \Gle\subseteq U^M_W$. The reverse inclusion is obvious.

Therefore $\Ga_k^e\big|E(R^M_e)$ is a homeomorphism of $E(R^M_e)$ onto an open set in
$\Gle$. From this it follows that $\dim E(R^M_e)=\dim \Gle$.
The proof  of the other result is similar.
\end{proof}

If $e\in E(M_n)$ then we have (see Proposition 3.1, Theorem 5.3 \cite{kris:00})
$$
\dim E(D_e)=\dim E(L_e)+\dim E(R_e).
$$
The next theorem is a generalization of this result.
\begin{thm}\label{Th.vii.5.4}
Let $e\in E(M)$. Then 
$$
\dim E(D^M_e)= \dim E(L^M_e)+\dim E(R^M_e).
$$
\end{thm}
\begin{proof}
Let 
\begin{align*}
Z^M_e & = \{(\rng[f],\nul[f])\,:\, f\in E(D^M_e)\} ,\\
Z'    & =\{(W,N)\in \Gla\times \Gra\,:\, W\oplus N=V\}.
\end{align*}
We first show that $Z_e^M=Z'$.

Obviously $Z^M_e\subseteq Z'$. Conversely, let $(W,N)\in Z'$. Let $a,b\in
D^M_e$ be such that $\rng[a]=W,\nul[b]=N$ and consider $H=L^M_a\cap R^M_b$.
Let $x,y\in H$. Now $\rng[x]=\rng[y]=W$ and $\nul[x]=\nul[y]=N$ and $W\oplus
N=V$. Hence $\rng[xy]=W$ and $\nul[xy]=N$. By Lemma \ref{Lemma.vii.4.1} we now
have $xy\in L^M_a\cap R^M_b =H$. Therefore $H$ is a subgroup of $M$
(see Theorem 2.16 \cite{Clif+:61}).
If $g$ is the idempotent in $H$ then, obviously, $g\in
E(D^M_e)$ and $W=\rng[g],N=\nul[g]$. This proves that $(W,N)\in Z^M_e$ showing
that $Z'\subseteq Z^M_e$.

Now, let $\Rank(e)=k$ and 
$$
Z_k=\{(W,N)\in \gk\times\gk[n-k]\,:\, W\oplus N=V\}.$$
By Theorem 5.6 \cite{kris:00}, $Z_k$ is an open subset of
$\gk\times\gk[n-k]$ and the map $\zeta:f\mapsto (\rng[f],\nul[f])$ is a
homeomorphism of $E(D_e)$ onto $Z_k$. The restriction $\zeta\big |E(D^M_e)$ is
a homeomorphism of $E(D^M_e)$ onto $Z^M_e$. 

Since $\Gla\subseteq \gk$ and $\Gra\subseteq \gk[n-k]$ the  topology on the
product space
$\Gla \times \Gra$ is the subspace topology inherited from the product space 
$\gk\times\gk[n-k]$. Since $Z_k$ is open in $\gk\times \gk[n-k]$ and since
$Z^M_e=Z_k\cap (\Gla\times \Gra)$, $Z^M_e$ is an open set in $\Gla\times \Gra$.
Thus $E(D^M_e)$ is homeomorphic to an open set in $\Gla\times\Gra$. From this
we have
$$%\begin{equation*}
\dim E(D^M_e)=\dim \Gla+\dim \Gra.
$$%\end{equation*}
 The theorem now follows from  Theorem
\ref{Th.vii.5.2}.
\end{proof}

\end{document}